\documentclass[11pt,a4paper,reqno]{amsart}
\usepackage[margin=30mm]{geometry}
\usepackage[T1]{fontenc}
\usepackage{lmodern}
\usepackage{amsmath,amssymb}
\usepackage{microtype}
\usepackage{needspace}
\usepackage[hidelinks]{hyperref}

\newtheorem{theorem}{Theorem}[section]
\newtheorem{proposition}[theorem]{Proposition}
\newtheorem{lemma}[theorem]{Lemma}
\newtheorem{corollary}[theorem]{Corollary}
\theoremstyle{definition}
\newtheorem{definition}[theorem]{Definition}
\newtheorem{example}[theorem]{Example}
\theoremstyle{remark}
\newtheorem{remark}[theorem]{Remark}
\numberwithin{equation}{section}
\AddToHook{env/theorem/before}{%
  \ifnum\value{section}=1\relax\Needspace{12\baselineskip}\fi}

\newcommand{\Z}{\mathbb Z}
\newcommand{\R}{\mathbb R}
\newcommand{\steps}{\mathcal S}
\newcommand{\ind}{\iota}
\DeclareMathOperator{\dom}{dom}
\DeclareMathOperator{\conv}{conv}

\newcommand{\ip}[2]{\langle #1,#2\rangle}
\newcommand{\normone}[1]{\lVert #1\rVert_1}
\newcommand{\normtwo}[1]{\lVert #1\rVert_2}

\title[Jensen's inequality on parity classes]{Jensen's inequality on parity classes of jump \texorpdfstring{$M$}{M}-convex functions}
\author[X. Liu]{Xinyue Liu}
\address{Tsinghua University, Beijing, China}
\email{liuxinyue743@gmail.com}

\date{10 September 2026}
\hypersetup{pdfauthor={Xinyue Liu},
  pdftitle={Jensen's inequality on parity classes of jump M-convex functions},
  pdfsubject={Preprint}}

\begin{document}
\begin{abstract}
We prove Jensen's inequality on each coordinatewise parity class of a
jump $M$-convex function on $\Z^n$, without assuming strong base orderability or a
finite effective domain. If a finite convex combination of domain points
has its barycenter in the same parity class, then the barycenter belongs
to the domain and satisfies Jensen's inequality. We also prove that every
domain point admits an affine support on its entire parity class. The set
of support vectors is characterized by at most $3^n-1$ local inequalities,
which yield optimality certificates, attained duality, and an exact
formula for linear growth and stability under linear perturbations.
The same conclusions hold for distributions with finite first moment
supported on domain points of a single parity class, provided their mean
lies in that class. We characterize equality by contact sets
and give a quantitative bound on the Jensen gap. The proof combines
induction on dimension with the known closure under infimal convolution:
convolution with an even-box indicator gives finite-domain supports, a quadratic
kernel proves barycenter feasibility, and a fixed local kernel yields
global supports. A two-dimensional example shows that total parity alone
does not suffice.

\end{abstract}
\maketitle

\section{Introduction}\label{sec:intro}

Jump $M$-convexity extends discrete exchange convexity to functions on
constant-parity jump systems. Introduced by Murota \cite{Murota2006}, this
class includes functions arising from minimum-weight perfect
$b$-matchings and separable convex costs on graph degree sequences. Its defining
exchange inequality compares two input points with two output points
obtained by complementary two-step moves. The class is closed under a
number of natural operations, including infimal convolution
\cite{KMT2007}.

The exchange inequality does not imply ordinary Jensen convexity at every
integer barycenter. The obstruction persists even if the inputs and the
barycenter have the same total coordinate parity and all belong to the
effective domain; see Example~\ref{ex:parity}. It is therefore natural to
ask on which lattice cosets Jensen's inequality holds.

Lassota and Ligthart \cite[Lemma~18]{LL2025} proved such an inequality for
strongly base-orderable (SBO) jump $M$-convex functions after fixing the
parity of each coordinate. They used the resulting convexity property
for matching value functions in a polyhedral formulation of integer
programs with additional variables. In their concluding discussion
\cite[Section~5]{LL2025}, they proposed extending the inequality to
general jump $M$-convex functions. The present paper proves this extension.

Throughout, $n$ is a positive integer. A function on a subset of $\Z^n$
is extended by $+\infty$ outside its domain. Convex combinations have
finite support; distributions of countably infinite support are treated
separately in Section~\ref{sec:probability}. Empty minima and infima have value $+\infty$, and
$\conv(\varnothing)=\varnothing$. For $r\in\{0,1\}^n$, put
\[
L_r=r+2\Z^n.
\]
We use jump $M$-convexity in the simultaneous two-output sense recalled in
Definition~\ref{def:jumpm}.

\begin{theorem}[Jensen inequality on a parity class]\label{thm:main}
Let $f:\Z^n\to\R\cup\{+\infty\}$ be a jump $M$-convex function with
nonempty effective domain $J$. Fix $r\in\{0,1\}^n$ and an integer $m\ge1$.
Let $x^1,\ldots,x^m\in J\cap L_r$, and let
$\lambda_1,\ldots,\lambda_m$ be nonnegative real numbers satisfying
\[
\sum_{j=1}^m\lambda_j=1,
\qquad
z=\sum_{j=1}^m\lambda_jx^j\in L_r.
\]
Then
\begin{equation}\label{eq:main}
z\in J,
\qquad
f(z)\le\sum_{j=1}^m\lambda_jf(x^j).
\end{equation}
\end{theorem}

In particular, membership of the barycenter in the effective domain is
a conclusion, rather than an additional hypothesis. The theorem applies
to infinite domains and arbitrary real convex coefficients.

Under the same assumptions on the weights and the barycenter, the
inequality also allows points $x^j\in L_r$ outside $J$.
After omitting zero-weight terms,
a value $f(x^j)=+\infty$ makes the right-hand side $+\infty$, so the
inequality is automatic but gives no feasibility conclusion. If all
remaining values are finite, Theorem~\ref{thm:main} applies.

The feasibility and value assertions have a common geometric
formulation. Write $J_r=J\cap L_r$ and
\[
E_r=\{(x,t)\in J_r\times\R:t\ge f(x)\}.
\]
Then Theorem~\ref{thm:main} implies
\begin{equation}\label{eq:intro-epigraph}
\conv(E_r)\cap(L_r\times\R)=E_r.
\end{equation}
Thus taking the convex hull of the epigraph introduces no new feasible
points and does not lower function values on the same parity class. The convex hull in this statement is the ordinary convex hull,
without taking its closure. This distinction allows infinite domains
without additional topological hypotheses.

Our proof proceeds through a finite-domain affine-support statement: at any
domain point, the function admits a common affine lower support on all
domain points of the same parity class. We prove this statement by
induction on dimension. A support on a coordinate slice determines a
linear perturbation; convolution with the indicator of a finite even
box then reduces the remaining support problem to one dimension.
This gives Jensen's inequality when the barycenter is already known to
be feasible. A second convolution, using a nonnegative quadratic
kernel with a unique zero, removes that preliminary feasibility
assumption.

The only non-elementary input to the proof of Theorem~\ref{thm:main}
is the convolution closure theorem of Kobayashi, Murota, and Tanaka
\cite{KMT2007}. Both functions in every convolution used below have finite effective
domains. All reductions, the support induction, and the feasibility
argument are proved below.

The finite-domain support theorem also yields a global conclusion that does
not require finiteness of the domain. If $z\in J_r$ and
$N_z=J\cap(z+\{-2,0,2\}^n)$, then a vector supports $f$ at $z$ on
all of $J_r$ if and only if it supports $f$ on $N_z$; such vectors
always exist (Theorem~\ref{thm:global-support}). Thus global supports
are specified by at most $3^n-1$ local linear inequalities. This leads
to an optimality criterion and an exact local formula for the sharp
growth of a linearly perturbed objective. Global supports also give
attained conjugate representations and extend Jensen's inequality,
including barycenter feasibility, to distributions with finite first
moment. Contact sets characterize equality, and the sharp growth
modulus bounds the Jensen gap quantitatively.

Section~\ref{sec:prelim} gives the definitions and elementary reductions.
Section~\ref{sec:support} proves the affine-support statement and Jensen's
inequality at feasible barycenters. Section~\ref{sec:feasibility} proves
the domain assertion and completes the main theorem.
Section~\ref{sec:global-support} proves the global support theorem and
its optimality and stability consequences. Section~\ref{sec:geometry}
develops exact convexification and attained duality, and shows why
coordinatewise parity is necessary. Section~\ref{sec:probability}
treats integrable distributions, equality, and quantitative Jensen bounds.

\section{Exchange convexity and finite reductions}\label{sec:prelim}

\subsection{The exchange axiom}

Let $e_1,\ldots,e_n$ be the standard basis of $\R^n$, and put
$\steps_n=\{\pm e_1,\ldots,\pm e_n\}$. For $x,y\in\Z^n$, a step $s\in\steps_n$ is
\emph{directed from $x$ to $y$} if
\[
\normone{x+s-y}=\normone{x-y}-1.
\]
For $s=\sigma e_i$, with $\sigma\in\{-1,1\}$, this is equivalent to
$\sigma(y_i-x_i)>0$. We also write
\[
[x,y]_{\Z}=\{w\in\Z^n:
\min(x_i,y_i)\le w_i\le\max(x_i,y_i)\text{ for every }i\}.
\]
A nonempty set $J\subseteq\Z^n$ has \emph{constant parity} if
$\sum_i x_i$ has the same parity for every $x\in J$.
This condition concerns total parity; it does not require $J$ to lie
in a single $L_r$.

\begin{samepage}
\begin{definition}\label{def:jumpm}
A function $f:\Z^n\to\R\cup\{+\infty\}$ is \emph{jump $M$-convex} if
its effective domain $J=\dom f=\{x:f(x)<+\infty\}$ is nonempty, has
constant parity, and satisfies the following exchange axiom.
For every $x,y\in J$ and every step $s$ directed from $x$ to $y$, there
exists a step $t$ directed from $x+s$ to $y$ such that
\begin{equation}\label{eq:exchange-domain}
x+s+t\in J,\qquad y-s-t\in J,
\end{equation}
and
\begin{equation}\label{eq:JM}
f(x)+f(y)\ge f(x+s+t)+f(y-s-t).
\end{equation}
\end{definition}
\end{samepage}

We will also use the corresponding set axiom, obtained by retaining
\eqref{eq:exchange-domain} and the quantifiers on $x,y,s,t$.
For $A\subseteq\Z^n$, its extended-valued indicator is
\[
\ind_A(x)=
\begin{cases}
0,&x\in A,\\
+\infty,&x\notin A.
\end{cases}
\]
If $A$ is nonempty, has constant parity, and satisfies the simultaneous
set exchange axiom, then $\ind_A$ is jump $M$-convex.

\subsection{Operations used in the proof}

\begin{lemma}\label{lem:reductions}
The following operations preserve jump $M$-convexity whenever the
resulting effective domain is nonempty:
translation of the argument by an integer vector; addition of a linear
function; fixing some coordinates; and restriction to an integer
coordinate box.
\end{lemma}

\begin{proof}
Translation preserves differences and directed steps. For linear
perturbations, the two exchange outputs have the same sum as the two
inputs, so the added linear terms cancel from \eqref{eq:JM}.

For a coordinate slice, the inputs agree in every fixed coordinate.
A directed first step cannot change any such coordinate. The second
step cannot change one either, since the first step left those
coordinates equal. Both outputs therefore remain in the slice.
Constant parity is inherited after subtracting the sum of the fixed
coordinates.

Finally, suppose that $x,y$ lie in an integer box $B$.
For any exchange supplied by \eqref{eq:JM}, directedness gives
$x'=x+s+t\in[x,y]_{\Z}$. Since $y'=y-s-t=x+y-x'$, we also have
$y'\in[x,y]_{\Z}$. The interval $[x,y]_{\Z}$ is contained in $B$.
Both outputs remain in the restricted domain, and their function
values are unchanged.
\end{proof}

For extended-valued functions $F,G$, define their infimal convolution by
\[
(F\mathbin{\Box}G)(w)=
\inf_{\substack{x,v\in\Z^n\\x+v=w}}\bigl(F(x)+G(v)\bigr).
\]
An empty set of finite candidates has infimum $+\infty$.
We use the following known theorem.

\begin{theorem}[Kobayashi--Murota--Tanaka]\label{thm:convolution}
If $F$ and $G$ are jump $M$-convex functions and
$F\mathbin{\Box}G$ never takes the value $-\infty$, then
$F\mathbin{\Box}G$ is jump $M$-convex.
\end{theorem}

This is the convolution closure theorem of~\cite{KMT2007}; see also
Theorem~12 in the authors' \href{https://www.keisu.t.u-tokyo.ac.jp/data/2006/METR06-13.pdf}{technical report METR~2006--13}.
It applies to the exchange axiom in
Definition~\ref{def:jumpm} and requires no SBO hypothesis.
In every application below, both effective domains are finite.
Consequently, each finite convolution value is an attained minimum,
the value $-\infty$ is impossible, and
\begin{equation}\label{eq:domain-convolution}
\dom(F\mathbin{\Box}G)=\dom F+\dom G.
\end{equation}

\subsection{One-dimensional supports and even boxes}

\begin{lemma}\label{lem:one-dimensional}
Let $q:\Z\to\R\cup\{+\infty\}$ be jump $M$-convex, and suppose that
its effective domain $D$ is finite, contained in $2\Z$, and contains
$0$. There exists $b\in\R$ such that
\begin{equation}\label{eq:one-dimensional-support}
q(k)\ge q(0)+bk\qquad(k\in D).
\end{equation}
\end{lemma}

\begin{proof}
If $a<b$ belong to $D$, the first and second steps directed from $a$
to $b$ must both be $+1$. Thus $a+2,b-2\in D$. Iterating shows that
\[
D=\{2j:A\le j\le B\}
\]
for integers $A\le0\le B$. For $A<j<B$, apply \eqref{eq:JM} to
$2j-2$ and $2j+2$ to obtain
\[
q(2j-2)+q(2j+2)\ge2q(2j).
\]
Hence the adjacent slopes
\[
d_j=\frac{q(2j+2)-q(2j)}2,\qquad A\le j<B,
\]
are nondecreasing.

If $A<0$, choose $b=d_{-1}$; if $A=0<B$, choose $b=d_0$; if
$A=B=0$, choose $b=0$. For $r>0$ with $2r\in D$, telescoping gives
\[
q(2r)-q(0)=2\sum_{j=0}^{r-1}d_j\ge2rb.
\]
For $r>0$ with $-2r\in D$, we have $A<0$ and
\[
q(0)-q(-2r)=2\sum_{j=-r}^{-1}d_j\le2rb.
\]
These inequalities prove \eqref{eq:one-dimensional-support},
including the endpoint and singleton cases.
\end{proof}

\begin{lemma}\label{lem:even-box}
Let $L_1,\ldots,L_n$ be nonnegative even integers, and put
\[
K=\prod_{i=1}^n([-L_i,L_i]\cap2\Z).
\]
Both $\ind_K$ and the function equal to $\normtwo{v}^2$ on $K$
and $+\infty$ outside $K$ are jump $M$-convex.
\end{lemma}

\begin{proof}
The set $K$ is nonempty, and every point in $K$ has even coordinates.
Let $a,b\in K$ and let $s=\sigma e_i$ be directed from $a$ to $b$.
The integer $d=\sigma(b_i-a_i)$ is positive and even, so $d\ge2$.
We may therefore choose $t=s$. Both outputs $a+2s$ and $b-2s$
belong to $K$: they are even and lie coordinatewise between $a$
and $b$. This proves the assertion for $\ind_K$.
For the quadratic function, all unchanged coordinates cancel and
\begin{align*}
&\normtwo{a}^2+\normtwo{b}^2
 -\normtwo{a+2s}^2-\normtwo{b-2s}^2\\
&\hspace{25mm}=4\sigma(b_i-a_i)-8=4(d-2)\ge0.
\end{align*}
The same choice of $t$ proves \eqref{eq:JM}. Coordinates with
$L_i=0$ admit no directed first step and require no separate argument.
\end{proof}

\section{Affine supports on finite domains}\label{sec:support}

The central inductive statement produces one support vector valid
simultaneously at every domain point in the parity class of the origin.

\begin{proposition}\label{prop:support}
Let $F:\Z^n\to\R\cup\{+\infty\}$ be jump $M$-convex with finite
effective domain $V$ containing $0$. There exists $p\in\R^n$ such that
\begin{equation}\label{eq:support}
F(v)\ge F(0)+\ip{p}{v}\qquad(v\in V\cap2\Z^n).
\end{equation}
\end{proposition}

\begin{proof}
We argue by induction on $n$. For $n=1$, constant parity and
$0\in V$ imply $V\subseteq2\Z$, so
Lemma~\ref{lem:one-dimensional} applies.

Suppose $n\ge2$ and the assertion holds in dimension $n-1$.
Write a point of $\Z^n$ as $(u,h)$ with
$u\in\Z^{n-1}$ and $h\in\Z$. The zero slice
\[
F_0(u)=F(u,0)
\]
is jump $M$-convex by Lemma~\ref{lem:reductions}, with finite
effective domain containing $0$. By the induction hypothesis,
there exists $a\in\R^{n-1}$ such that
\begin{equation}\label{eq:slice-support}
F(u,0)-\ip{a}{u}\ge F(0)
\quad\text{if }(u,0)\in V,\quad u\in2\Z^{n-1}.
\end{equation}
Define
\[
\widetilde F(u,h)=F(u,h)-\ip{a}{u}
\]
on $V$, and set it equal to $+\infty$ outside $V$.
This is a jump $M$-convex function with the same effective domain.

Choose nonnegative even integers $L_1,\ldots,L_{n-1}$ satisfying
\[
L_i\ge\max_{v\in V}|v_i|\qquad(1\le i<n),
\]
and form the finite horizontal even box
\[
K=\left(\prod_{i=1}^{n-1}([-L_i,L_i]\cap2\Z)\right)\times\{0\}.
\]
By Lemma~\ref{lem:even-box}, its indicator is jump $M$-convex.
Theorem~\ref{thm:convolution} now shows that
\[
H=\widetilde F\mathbin{\Box}\ind_K
\]
is jump $M$-convex, with finite effective domain $V+K$ containing $0$.

Restrict $H$ to the last coordinate axis:
\[
q(h)=H(0,h).
\]
This is a one-dimensional jump $M$-convex function by
Lemma~\ref{lem:reductions}. Since $0\in V$, all points of $V$
have even total parity; the same is true of $V+K$.
An axis point $(0,h)\in V+K$ therefore has $h\in2\Z$.
Thus $\dom q$ is finite, contained in $2\Z$, and contains $0$.

On this axis, the convolution formula simplifies as follows.
A decomposition with finite cost of $(0,h)$ must be
$(0,h)=(u,h)+(-u,0)$, with $(u,h)\in V$ and $(-u,0)\in K$.
The bounds $L_i$ include the horizontal coordinates of every point in $V$.
It follows that, for every $h\in\Z$,
\begin{equation}\label{eq:axis-formula}
q(h)=
\min_{\substack{(u,h)\in V\\u\in2\Z^{n-1}}}
\bigl(F(u,h)-\ip{a}{u}\bigr),
\end{equation}
with value $+\infty$ when there is no candidate.
For $h=0$, every candidate is at least $F(0)$ by
\eqref{eq:slice-support}, while $u=0$ is a candidate of value $F(0)$.
Hence
\begin{equation}\label{eq:anchor}
q(0)=F(0).
\end{equation}

Lemma~\ref{lem:one-dimensional} supplies $b\in\R$ such that
$q(h)\ge q(0)+bh$ for every $h\in\dom q$.
If $v=(u,h)\in V\cap2\Z^n$, then $(u,h)$ is a candidate in
\eqref{eq:axis-formula}, so $h\in\dom q$. Consequently,
\[
F(u,h)-\ip{a}{u}\ge q(h)\ge F(0)+bh.
\]
Taking $p=(a,b)$ proves \eqref{eq:support}.
The vector $a$ was fixed by the entire zero slice, and $b$ by the
entire axis function; neither depends on the final choice of $v$.
\end{proof}

\begin{proposition}[Jensen inequality at feasible barycenters]\label{prop:anchored}
Let $f:\Z^n\to\R\cup\{+\infty\}$ be jump $M$-convex with effective
domain $J$, and let $m\ge1$ be an integer. Suppose $z,x^1,\ldots,x^m\in J$,
$x^j-z\in2\Z^n$, and $\lambda_1,\ldots,\lambda_m\ge0$ satisfy
$\sum_j\lambda_j=1$ and $\sum_j\lambda_jx^j=z$. Then
$f(z)\le\sum_j\lambda_jf(x^j)$.
\end{proposition}

\begin{proof}
Translate by $z$ and put $d^j=x^j-z\in2\Z^n$.
Choose a finite integer coordinate box $B$ containing $0$ and all
the $d^j$. Let $F$ agree with $w\mapsto f(z+w)$ on $B$ and
take the value $+\infty$ outside $B$. By
Lemma~\ref{lem:reductions}, $F$ is jump $M$-convex with a finite
effective domain containing $0,d^1,\ldots,d^m$.
Proposition~\ref{prop:support} yields $p\in\R^n$ with
\[
f(x^j)\ge f(z)+\ip{p}{x^j-z}\qquad(1\le j\le m).
\]
Multiplying by the original real coefficients and summing gives
\[
\sum_j\lambda_jf(x^j)
\ge f(z)+\ip{p}{\sum_j\lambda_jx^j-z}
=f(z).
\]
No rational approximation of the coefficients is required.
\end{proof}

\begin{remark}\label{rem:finite-support}
The preceding proof chooses a support in a finite box and does not
need uniform control as the box varies. Section~\ref{sec:global-support}
establishes the stronger assertion that one vector supports the entire
parity class, by showing that a fixed finite neighborhood determines
all global supports.
\end{remark}

\section{Feasibility of same-parity barycenters}\label{sec:feasibility}

The next step proves feasibility by applying
Proposition~\ref{prop:anchored} to an auxiliary function whose
effective domain is known to contain the barycenter.

\begin{proposition}\label{prop:feasibility}
Let $J\subseteq\Z^n$ be nonempty, have constant parity, and satisfy
the simultaneous set exchange axiom. Fix $r\in\{0,1\}^n$ and an integer
$m\ge1$. Suppose
$x^1,\ldots,x^m\in J\cap L_r$ and
\[
\lambda_j\ge0,\qquad \sum_j\lambda_j=1,\qquad
z=\sum_j\lambda_jx^j\in L_r.
\]
Then $z\in J$.
\end{proposition}

\begin{proof}
Delete any zero-weight terms; at least one term remains.
Take a finite integer coordinate box $B$ containing $z$ and all
the $x^j$, and set $V=J\cap B$. The set $V$ is nonempty and finite.
The interval argument of Lemma~\ref{lem:reductions} shows that
$V$ retains the simultaneous set exchange axiom.
Thus $\ind_V$ is jump $M$-convex.

For each coordinate set
\[
L_i=\max_{1\le j\le m}|z_i-x_i^j|,
\qquad
K=\prod_{i=1}^n([-L_i,L_i]\cap2\Z).
\]
The $L_i$ are nonnegative even integers.
Both $0$ and all displacements $z-x^j$ belong to $K$.
Define
\[
\psi(k)=
\begin{cases}
\normtwo{k}^2,&k\in K,\\
+\infty,&k\notin K.
\end{cases}
\]
By Lemma~\ref{lem:even-box}, $\psi$ is jump $M$-convex.
It is nonnegative and vanishes only at $0$.

Consider
\[
G=\ind_V\mathbin{\Box}\psi,
\qquad
G(w)=\min_{\substack{v\in V,\ k\in K\\v+k=w}}\normtwo{k}^2.
\]
Both functions have finite nonempty effective domains, so
Theorem~\ref{thm:convolution} applies.
The resulting jump $M$-convex function has effective domain $V+K$.
It is nonnegative. Moreover,
\begin{equation}\label{eq:zero-set}
G(w)=0\quad\Longleftrightarrow\quad w\in V.
\end{equation}
Indeed, a point $w\in V$ has the decomposition $w=w+0$.
Conversely, a minimum equal to zero is attained, and its quadratic
term forces $k=0$, hence $w=v\in V$.

In particular, $G(x^j)=0$ for every $j$.
The explicit decomposition
\[
z=x^1+(z-x^1)
\]
has its first term in $V$ and its second in $K$, so
$z\in\dom G$. This assertion uses no membership assumption on
$z$ in the original set $J$.
We may therefore apply Proposition~\ref{prop:anchored} to $G$,
the same points, and the same coefficients. It gives
\[
0\le G(z)\le\sum_j\lambda_jG(x^j)=0.
\]
Equation~\eqref{eq:zero-set} now yields $z\in V\subseteq J$.
\end{proof}

\begin{proof}[Proof of Theorem~\ref{thm:main}]
The effective domain of $f$ satisfies the set assumptions of
Proposition~\ref{prop:feasibility}. That proposition gives $z\in J$.
Proposition~\ref{prop:anchored}, applied to the original function
$f$, then gives the required inequality.
\end{proof}

\section{Global supports and local optimality}\label{sec:global-support}

The finite-domain support statement also gives a finite description of
the set of global support vectors, even when the effective domain is
infinite. Put
\[
K=\{-2,0,2\}^n.
\]
We first establish a local optimality criterion without restricting a
jump $M$-convex function to a parity class as an intermediate operation.

\begin{lemma}[Local optimality on a parity class]\label{lem:parity-local}
Let $g$ be jump $M$-convex with $0\in\dom g$. If
\[
g(u)\ge g(0)\qquad(u\in\dom g\cap K),
\]
then $g(x)\ge g(0)$ for every $x\in\dom g\cap2\Z^n$.
\end{lemma}

\begin{proof}
Fix such an $x$, and choose a finite integer coordinate box $B$
containing $[-2,2]^n\cap\Z^n$ and $x$. Let $G$ agree with $g$ on
$B$ and be $+\infty$ outside $B$. By Lemma~\ref{lem:reductions},
$G$ is jump $M$-convex with finite effective domain. Lemma~\ref{lem:even-box}
and Theorem~\ref{thm:convolution} imply that
\[
H(w)=(G\mathbin{\Box}\ind_K)(w)=\min_{k\in K}G(w-k)
\]
is also jump $M$-convex with finite effective domain.
Since $K=-K$, the assumed local inequalities and the candidate $k=0$
give $H(0)=g(0)$. The candidates $k=2e_i$ and $k=-2e_i$ give
\[
H(2e_i)\le g(0),\qquad H(-2e_i)\le g(0)
\qquad(1\le i\le n).
\]
In particular, these $2n+1$ points belong to $\dom H\cap2\Z^n$.
By Proposition~\ref{prop:support}, there exists $a\in\R^n$ with
\[
H(w)\ge H(0)+\ip{a}{w}
\qquad(w\in\dom H\cap2\Z^n).
\]
At $w=\pm2e_i$, this implies $2a_i\le0$ and $-2a_i\le0$.
Hence $a=0$. Finally, using the candidate $k=0$ at $w=x$ gives
\[
g(x)=G(x)\ge H(x)\ge H(0)=g(0).
\]
The point $x$ was arbitrary.
\end{proof}

For $z\in J_r=\dom f\cap L_r$, define
\begin{equation}\label{eq:parity-subdiff}
\partial_r f(z)=
\{p\in\R^n:f(x)\ge f(z)+\ip{p}{x-z}\text{ for every }x\in J_r\},
\end{equation}
and put $N_z=\dom f\cap(z+K)$.
Thus $N_z$ consists of the effective points in the parity class of $z$
whose coordinates differ from those of $z$ by at most two.

\begin{theorem}[Finite characterization of global supports]\label{thm:global-support}
Let $f$ be jump $M$-convex, let $r\in\{0,1\}^n$, and let $z\in J_r$.
Then
\begin{equation}\label{eq:local-subdiff}
\partial_r f(z)=
\{p\in\R^n:\ip{p}{u-z}\le f(u)-f(z)\text{ for every }u\in N_z\}
\ne\varnothing.
\end{equation}
In particular, $\partial_r f(z)$ is a nonempty polyhedron defined by at
most $3^n-1$ nontrivial inequalities, regardless of the size of $\dom f$.
\end{theorem}

\begin{proof}
Apply Proposition~\ref{prop:support} to the restriction of
$x\mapsto f(z+x)$ to $[-2,2]^n\cap\Z^n$. This proves that the
right side of \eqref{eq:local-subdiff} is nonempty.
For any $p$ in that set, define
\[
g(x)=f(z+x)-\ip{p}{x}.
\]
This function is jump $M$-convex. The local inequalities give
$g(u)\ge g(0)$ for $u\in\dom g\cap K$.
By Lemma~\ref{lem:parity-local}, the inequality extends to
$\dom g\cap2\Z^n$, so $p\in\partial_r f(z)$.
The reverse inclusion follows by restricting the global
inequalities to $N_z$.
\end{proof}

The vector $p$ is selected in a fixed neighborhood of $z$ before any
distant point is considered. The proof therefore does not pass to a
limit of supports chosen in expanding boxes. The support may depend
on $z$, and it need not support points in other parity classes.

\begin{corollary}[Optimality certificates]\label{cor:optimality}
For $c\in\R^n$ and $z\in J_r$, the following are equivalent:
\begin{enumerate}
\item $z$ minimizes $f(x)-\ip{c}{x}$ over $J_r$;
\item $c\in\partial_r f(z)$;
\item $f(u)-\ip{c}{u}\ge f(z)-\ip{c}{z}$ for every $u\in N_z$.
\end{enumerate}
\end{corollary}

\begin{proof}
The first two statements are equivalent by definition, and the last
two by Theorem~\ref{thm:global-support}.
\end{proof}

Given exact domain-membership and value queries and the value $f(z)$,
the last condition requires at most $3^n-1$ neighbor-membership queries
and at most that many neighboring value queries. A support can be
found by solving the linear feasibility system \eqref{eq:local-subdiff}
in $n$ variables; rational local values admit a rational support.
For $n=3$, the system has at most $26$ nontrivial inequalities, including
diagonal directions. These are bounds on the number of local queries,
not on the cost of evaluating a particular representation of $f$.

The criterion also applies after imposing integer coordinate bounds,
since such restrictions preserve jump $M$-convexity. On a finite nonempty $J_r$, this gives a descent procedure for
$f(x)-\ip{c}{x}$. At each iteration, let $z$ be the current point and
move to a strictly improving point of $N_z$, if one exists. The procedure
terminates at a global minimizer after at most $|J_r|-1$ moves. This
cardinality bound does not imply polynomial running time in the input
encoding length.

\subsection{Sharp growth and stability}

The same finite system determines the exact linear growth away from
a supported point. For $p\in\partial_r f(z)$, write
\[
D_p(x)=f(x)-f(z)-\ip{p}{x-z}\qquad(x\in J_r).
\]
This quantity is nonnegative.

\Needspace{21\baselineskip}
\begin{proposition}[Exact growth modulus]\label{prop:sharpness}
Suppose $J_r\ne\{z\}$ and $p\in\partial_r f(z)$. Then
$N_z\setminus\{z\}\ne\varnothing$ and
\begin{equation}\label{eq:sharp-modulus}
\inf_{x\in J_r\setminus\{z\}}
\frac{D_p(x)}{\normone{x-z}}
=\min_{u\in N_z\setminus\{z\}}
\frac{D_p(u)}{\normone{u-z}}
=:\alpha_z(p)\ge0.
\end{equation}
Moreover, $\alpha_z(p)$ is the largest $\varepsilon\ge0$ for which
\begin{equation}\label{eq:support-stability}
p+[-\varepsilon,\varepsilon]^n\subseteq\partial_r f(z).
\end{equation}
The bound below also holds when $\alpha_z(p)=0$:
\begin{equation}\label{eq:sharp-growth}
f(x)\ge f(z)+\ip{p}{x-z}+\alpha_z(p)\normone{x-z}
\qquad(x\in J_r).
\end{equation}
In particular, $z$ is the unique minimizer of $f(x)-\ip{p}{x}$ on
$J_r$ if and only if it is strictly better than every other point in
$N_z$. In this case $\alpha_z(p)>0$.
\end{proposition}

\begin{proof}
If $N_z=\{z\}$, Theorem~\ref{thm:global-support} would give
$\partial_r f(z)=\R^n$. A point $x\in J_r\setminus\{z\}$ would
then violate the support inequality for a sufficiently large positive
multiple of $x-z$. Thus the displayed local minimum exists.

For $\varepsilon\ge0$, condition \eqref{eq:support-stability} is,
by Theorem~\ref{thm:global-support}, equivalent to
\[
D_p(u)\ge
\max_{\|h\|_\infty\le\varepsilon}\ip{h}{u-z}
=\varepsilon\normone{u-z}
\qquad(u\in N_z).
\]
It therefore holds exactly when $\varepsilon\le\alpha_z(p)$.
Using the global definition \eqref{eq:parity-subdiff} instead shows
that the same condition is equivalent to
$D_p(x)\ge\varepsilon\normone{x-z}$ for every $x\in J_r$.
This proves \eqref{eq:sharp-modulus} and \eqref{eq:sharp-growth}.
Strict local inequalities make the minimum of the finitely many
positive ratios positive, and hence imply global uniqueness.
Global uniqueness plainly implies the strict local inequalities.
\end{proof}

If $\alpha_z(p)>0$ and $\|h\|_\infty<\alpha_z(p)$, then $z$ remains
the unique minimizer after replacing $p$ by $p+h$. Indeed, the new gap
is at least $(\alpha_z(p)-\|h\|_\infty)\normone{x-z}$.
Thus the finite local system gives an exact radius for preserving
optimality under linear perturbations, as well as a quantitative
certificate of uniqueness.

\section{Exact convexification and the parity condition}\label{sec:geometry}

\subsection{Domains and epigraphs}

\begin{corollary}\label{cor:domain}
For the domain $J$ of a jump $M$-convex function and every
$r\in\{0,1\}^n$,
\[
\conv(J_r)\cap L_r=J_r,\qquad J_r=J\cap L_r.
\]
\end{corollary}

\begin{proof}
If $z$ belongs to the left side, it is a finite convex combination
of points in $J_r$ and lies in $L_r$.
Proposition~\ref{prop:feasibility} implies $z\in J_r$.
The reverse inclusion is immediate.
\end{proof}

\begin{corollary}\label{cor:epigraph}
For $E_r=\{(x,t)\in J_r\times\R:t\ge f(x)\}$, one has
\[
\conv(E_r)\cap(L_r\times\R)=E_r.
\]
\end{corollary}

\begin{proof}
Suppose $(z,t)=\sum_j\lambda_j(x^j,t_j)$ with
$(x^j,t_j)\in E_r$ and $z\in L_r$.
Theorem~\ref{thm:main} gives $z\in J_r$ and
\[
f(z)\le\sum_j\lambda_jf(x^j)
\le\sum_j\lambda_jt_j=t.
\]
Thus $(z,t)\in E_r$. The other inclusion is immediate.
\end{proof}

For finite nonempty $J_r$, its convexified epigraph is the polyhedron
\[
P_r=\conv\{(x,f(x)):x\in J_r\}
    +\{(0,s):s\ge0\}.
\]
Corollary~\ref{cor:epigraph} states that $P_r$ represents the function
exactly at points of $L_r$. If $f(x)$ is rational for every $x\in J_r$,
this is a rational polyhedron.
For an infinite $J_r$, the corollary still holds with the ordinary
convex hull; no polyhedrality or closedness is asserted.

\subsection{The convex-combination formula}

Another formulation is obtained by minimizing over all representations
of a prescribed barycenter. For $z\in L_r$, define
\begin{equation}\label{eq:convex-envelope}
\begin{split}
\mathcal C_r f(z)=\inf\Bigl\{
&\sum_{j=1}^m\lambda_jf(x^j):
m\ge1,\ x^j\in J_r,\ \lambda_j\ge0,\\
&\sum_j\lambda_j=1,\quad
\sum_j\lambda_jx^j=z
\Bigr\},
\end{split}
\end{equation}
where the infimum of the empty set is $+\infty$.

\begin{corollary}\label{cor:envelope}
For every $z\in L_r$, one has $\mathcal C_r f(z)=f(z)$.
\end{corollary}

\begin{proof}
If $z\in J_r$, Theorem~\ref{thm:main} bounds each candidate in
\eqref{eq:convex-envelope} below by $f(z)$, and the singleton
combination $z$ attains that value. If $z\notin J_r$,
Corollary~\ref{cor:domain} rules out every candidate, so both
sides are $+\infty$.
\end{proof}

These conclusions isolate the convex-geometric assertion from the
algorithmic operations needed to exploit it. In the matching setting
of \cite[Section~2.2]{LL2025}, additional structure supplies the relevant optimization
and separation procedures. Theorem~\ref{thm:main} removes the SBO
assumption from the Jensen property itself; it does not supply an
efficient separation oracle for an arbitrary representation of a
jump $M$-convex function.

\subsection{Attained duality and separating inequalities}

Fix a nonempty $J_r$. Define the conjugate by
\[
f_r^*(p)=\sup_{x\in J_r}\{\ip{p}{x}-f(x)\}
\qquad(p\in\R^n).
\]
This takes values in $\R\cup\{+\infty\}$, since $J_r$ is nonempty.

\begin{corollary}[Attained dual representation at domain points]\label{cor:dual}
For every $z\in J_r$,
\begin{equation}\label{eq:attained-dual}
f(z)=\max_{p\in\R^n}\{\ip{p}{z}-f_r^*(p)\}.
\end{equation}
The set of maximizers is exactly $\partial_r f(z)$.
For each such maximizer,
$f_r^*(p)=\ip{p}{z}-f(z)$, and the supremum defining $f_r^*(p)$
is attained at $x=z$.
\end{corollary}

\begin{proof}
The candidate $x=z$ in the conjugate gives
$\ip{p}{z}-f_r^*(p)\le f(z)$ for every $p$.
Equality holds precisely when
$\ip{p}{x}-f(x)\le\ip{p}{z}-f(z)$ for every $x\in J_r$,
which is the condition $p\in\partial_r f(z)$.
This set is nonempty by Theorem~\ref{thm:global-support}.
\end{proof}

In particular, for any known feasible $z\in J_r$ and any $t<f(z)$,
a support $p$ computed from \eqref{eq:local-subdiff} gives the valid cut
\[
s\ge f(z)+\ip{p}{x-z}
\]
for $\conv(E_r)$ that strictly separates $(z,t)$.
This separates a point below the epigraph whose lattice coordinate is
known to be feasible; it does not by itself decide domain membership or separate
arbitrary real query points.

For example, if $f$ is a jump $M$-convex matching value function as
in \cite[Section~2.2]{LL2025}, such a cut is a valid linear lower bound
on the cost of every feasible degree vector in the prescribed parity
class. Its construction uses local value and feasibility queries.
The same certificate applies to any other value function for which
the exchange axiom and these queries are available.

\medskip
\noindent\begin{minipage}{\linewidth}
\setlength{\parindent}{22pt}
\subsection{Total parity is not enough}

\begin{example}\label{ex:parity}
Let
\[
J=\{(0,0),(1,1),(1,-1),(-1,1),(-1,-1)\}\subseteq\Z^2,
\]
and define $f(0,0)=1$ and $f(x)=0$ at the other four points.
Extend $f$ by $+\infty$ outside $J$.
Then $f$ is jump $M$-convex, but
\[
f(0,0)>
\tfrac12 f(1,1)+\tfrac12 f(-1,-1).
\]
\end{example}
\end{minipage}

\begin{proof}
All points of $J$ have even total parity.
Distinct inputs have $\ell_1$ distance either $2$ or $4$.
In the distance-$2$ case, after any directed first step $s$ the
remaining difference $t=y-x-s$ is a directed unit step.
The outputs are $y$ and $x$, so \eqref{eq:JM} holds with equality.
In the distance-$4$ case, the inputs are opposite corner points.
For any directed first step $s$, choose $t=s$.
The outputs are the two remaining corners, and all four function
values in \eqref{eq:JM} are zero.
Equal inputs have no exchange obligation.
This verifies the full exchange axiom.

Finally, $(0,0)=\tfrac12(1,1)+\tfrac12(-1,-1)$, whereas
$f(0,0)=1$ and both source values are zero.
The inputs and barycenter have the same total parity and lie in
$J$, but their coordinatewise parity classes differ.
\end{proof}

\section{Integrable distributions and equality}\label{sec:probability}

Global supports extend Theorem~\ref{thm:main} to distributions with
countably infinite support. We first prove the required fact about means,
using standard separation and relative-interior results
\cite[Sections~6 and~11]{Rockafellar1970}.

\begin{lemma}\label{lem:mean-conv}
Let $X$ be a random vector taking values in a countable subset of
$\R^n$ almost surely, with $\mathbb E\normone{X}<\infty$. If
$S=\{x:\mathbb P(X=x)>0\}$, then $\mathbb EX\in\conv(S)$.
\end{lemma}

\begin{proof}
Then $\mathbb P(X\in S)=1$. Write $z=\mathbb EX$ and
$C=\conv(S)$. Separation from the closed
convex set $\overline C$ shows that $z\in\overline C$:
otherwise a separating linear functional would contradict its
expectation at $z$. If $z\notin C$, then $z$ lies on the relative
boundary of $\overline C$, because the relative interiors of $C$
and $\overline C$ coincide and are contained in $C$.
A proper supporting hyperplane to $\overline C$ at $z$, taken in
its affine hull, gives a vector $a$ with $\ip{a}{x-z}\ge0$ for
every $x\in S$. At least one inequality is strict: otherwise $S$,
and hence the affine hull of $C$, would lie in that proper hyperplane.
Every point of $S$ has positive
probability, so $\mathbb E\ip{a}{X-z}>0$, contrary to
$\mathbb EX=z$. The functional is integrable by the first-moment
assumption. Thus $z\in C$.
\end{proof}

\begin{theorem}[Jensen inequality for integrable distributions]\label{thm:integrable}
Let $f$ be jump $M$-convex, let $r\in\{0,1\}^n$, and let $X$ take
values in $J_r$ almost surely. Suppose
\[
\mathbb E\normone{X}<\infty,\qquad z=\mathbb EX\in L_r.
\]
Then $z\in J_r$, the negative part of $f(X)$ is integrable, and
\begin{equation}\label{eq:integrable-jensen}
f(z)\le\mathbb E f(X),
\end{equation}
where the right side may equal $+\infty$.
\end{theorem}

\begin{proof}
The random vector takes values in the countable set $J_r$ almost surely.
Lemma~\ref{lem:mean-conv} expresses $z$ as a finite convex combination
of points in $J_r$. Proposition~\ref{prop:feasibility} therefore
gives $z\in J_r$.
Take $p\in\partial_r f(z)$, which exists by
Theorem~\ref{thm:global-support}. The global inequality
\[
f(X)\ge f(z)+\ip{p}{X-z}
\]
bounds $f(X)$ below by an integrable random variable. The negative part of $f(X)$ is therefore integrable. Taking expectations gives
\eqref{eq:integrable-jensen}.
\end{proof}

\Needspace{18\baselineskip}
\begin{corollary}[Equality and a quantitative Jensen bound]\label{cor:jensen-gap}
Under the hypotheses of Theorem~\ref{thm:integrable}, fix
$p\in\partial_r f(z)$ and define the contact set
\[
C_{z,p}=\{x\in J_r:f(x)=f(z)+\ip{p}{x-z}\}.
\]
Then
\begin{equation}\label{eq:gap-identity}
\mathbb E f(X)-f(z)=\mathbb E D_p(X),
\end{equation}
as an identity in $[0,+\infty]$. Equality in Jensen holds if and only
if $\mathbb P(X\in C_{z,p})=1$.
If $J_r\ne\{z\}$, then
\begin{equation}\label{eq:quantitative-jensen}
\mathbb E f(X)-f(z)
\ge\alpha_z(p)\,\mathbb E\normone{X-z},
\end{equation}
where $\alpha_z(p)$ is determined by the finite formula
\eqref{eq:sharp-modulus}. If $\alpha_z(p)>0$, equality in Jensen
is possible only when $X=z$ almost surely.
\end{corollary}

\begin{proof}
The decomposition
$f(X)=f(z)+\ip{p}{X-z}+D_p(X)$ consists of an integrable affine
term and a nonnegative term. The affine term has expectation $f(z)$,
which proves \eqref{eq:gap-identity}. A nonnegative random variable
has expectation zero exactly when it vanishes almost surely.
If $J_r\ne\{z\}$, \eqref{eq:sharp-modulus} gives
$D_p(x)\ge\alpha_z(p)\normone{x-z}$ for every $x\in J_r$,
including when $\alpha_z(p)=0$.
Taking expectations proves \eqref{eq:quantitative-jensen}.
\end{proof}

Thus, among all integrable distributions in a fixed parity class with
prescribed feasible mean $z$, the minimum expected value is $f(z)$,
attained by the point mass at $z$. The contact set describes every
distribution attaining that value. When $\alpha_z(p)>0$, any finite
Jensen gap $\eta$ gives the stability estimate
\[
\mathbb E\normone{X-z}\le\eta/\alpha_z(p).
\]
These statements concern distributions within one coordinatewise
parity class; Example~\ref{ex:parity} precludes the corresponding
conclusion under total parity alone.

\section*{Use of AI tools}

The author initiated and directed the research and determined its
objectives and scope. OpenAI's ChatGPT and Codex were used under the
author's direction to assist with proof development, checking arguments
and references, and manuscript preparation. The author takes
responsibility for the final content.

\bibliographystyle{amsplain}
\bibliography{references}
\end{document}